\documentclass{amsart}
\theoremstyle{plain}
\newtheorem{theorem}{Theorem}[section]

\newtheorem{corollary}[theorem]{Corollary}

\newtheorem{lemma}[theorem]{Lemma}
\newtheorem{proposition}[theorem]{Proposition}

\theoremstyle{definition}

\theoremstyle{remark}
\newtheorem{remark}{Remark}[section]

\newtheorem{observation}[remark]{Observation}

\newcommand{\bbG}{\mathbb{G}}
\newcommand{\bbP}{\mathbb{P}}

\newcommand{\calA}{\mathcal{A}}
\newcommand{\calB}{\mathcal{B}}
\newcommand{\calE}{\mathcal{E}}

\newcommand{\calR}{\mathcal{R}}
\newcommand{\calU}{\mathcal{U}}
\newcommand{\calX}{\mathcal{X}}

\newcommand{\fg}{\mathfrak{g}}
\newcommand{\fh}{\mathfrak{h}}

\newcommand{\whcB}{\widehat{\mathcal{B}}}

\newcommand{\ol}{\overline}

\newcommand{\defn}[1]{{\bf \textcolor{blue}{#1}}}

\usepackage{hyperref}
\usepackage{amsthm,mathtools}
\usepackage[all]{xy}
\usepackage{mathrsfs}
\usepackage[dvipsnames]{xcolor}
\usepackage{amssymb,amsmath,tabularx,graphicx,float,placeins}
\usepackage{tabu} 
\usepackage{makecell} 

\usepackage{tikz}
\usetikzlibrary[calc,intersections,through,backgrounds,arrows,decorations.pathmorphing]

\begin{document}

\title{Some Combinatorial Formulas Related to Diagonal Ramsey Numbers}
\author{Pakawut Jiradilok}
\date{\today}

\address{Department of Mathematics, Massachusetts Institute of Technology, Cambridge, MA 02139}
\email[P.~Jiradilok]{pakawut@mit.edu}

\begin{abstract}
We derive some combinatorial formulas related to the diagonal Ramsey numbers $R(k)$. Each formula is a statement of the form ``$F(n,k) = 0$ if and only if $n \ge R(k)$,'' where $F(n,k)$ is a combinatorial expression which depends on $n$ and $k$. Our work is closely related to a recent work by De~Loera and Wesley.
\end{abstract}


\maketitle

\section{Introduction}
\subsection{Diagonal Ramsey Numbers}
The $k^{\text{th}}$ diagonal Ramsey number $R(k)$ is the smallest positive integer $n$ with the property that for any coloring of the edges of the complete graph on $n$ vertices with two colors, there exists a monochromatic clique of size $k$ (see e.g.~\cite{Erd90, Rad21}). The existence of $R(k)$ (as a finite positive integer) is a consequence of a theorem of Ramsey's~\cite{Ram29}. The classical result of Erd\H{o}s and Szekeres~\cite{ES35} gives the upper bound
\begin{equation}\label{ineq:R-k-ES}
R(k) \le \binom{2k-2}{k-1},
\end{equation}
which is referred to as the {\em Erd\H{o}s--Szekeres bound} in the literature (see e.g.~\cite{Thom88, Sah23, CGMS23}).

There have been exciting developments on the diagonal Ramsey numbers in recent years. In 1986, V.~R\"{o}dl, in an unpublished manuscript (as reported by Andrew~Thomason in~\cite{Thom88}), obtained an improvement of the upper bound in~\eqref{ineq:R-k-ES} by a logarithmic factor. There is also a weaker bound which improves~\eqref{ineq:R-k-ES} by a double-logarithmic (``$\log\log$'') factor by Graham and R\"odl \cite{GR87}. In~\cite{Thom88}, Thomason obtained a polynomial improvement of~\eqref{ineq:R-k-ES}, which was later strengthened by further breakthrough works establishing superpolynomial improvements of~\eqref{ineq:R-k-ES} by Conlon~\cite{Con09} and by Sah~\cite{Sah23}.

Approximately nine decades after Erd\H{o}s and Szekeres'~\cite{ES35}, Campos, Griffiths, Morris, and Sahasrabudhe in their 2023 preprint~\cite{CGMS23} obtained, for the first time, an exponential improvement to~\eqref{ineq:R-k-ES}. They proved that there exists an absolute constant $\varepsilon > 0$ such that for all sufficiently large $k$, we have the upper bound $R(k) \le (4-\varepsilon)^k$.

For the lower bound, we have not seen much improvement since 1947. The classical result of Erd\H{o}s'~\cite{Erd47} shows $R(k,k) \ge (1+o(1)) \cdot (k/e) \cdot 2^{\frac{k-1}{2}}$. Spencer later improved this to $R(k,k) \ge (1+o(1)) \cdot (k/e) \cdot 2^{\frac{k+1}{2}}$ (see \cite{Spe75, Spe77}). To our knowledge, no significant improvements have happened afterwards.

\medskip

In the present work, we take an algebraic combinatorics approach to studying the diagonal Ramsey numbers. We derive four formulas, each of which is a statement ``$F(n,k) = 0$ if and only if $n \ge R(k)$,'' where $F(n,k)$ is a combinatorial expression. We refer to Section~\ref{sec:defns-main-results} for the descriptions of our main results. For our first two formulas (see Theorems~\ref{thm:cos-mult} and~\ref{thm:cos-incidence} below), the expression $F(n,k)$ is a combinatorial summation of trigonometric functions.

Our third formula (see Theorem~\ref{thm:P-n-k} and Corollary~\ref{cor:P-nk-1-2} below) gives a univariate integral polynomial $P_{n,k}(t) \in \mathbb{Z}[t]$ with the property that $P_{n,k}(1/2) = 0$ if and only if $n \ge R(k)$. Similarly, our fourth (see Theorem~\ref{thm:Q-n-k} and Corollary~\ref{cor:Q-n-k} below) gives $Q_{n,k}(t) \in \mathbb{Z}[t]$ with the property that $Q_{n,k}(1-2^{(k^2-k-2)/2}) = 0$ if and only if $n \ge R(k)$.

We remark that our first three formulas are applicable for $n$ and $k$ under certain number theoretic assumptions. For the first two formulas, we also establish a generalization (Theorem~\ref{thm:cos-alpha-beta}) which applies for all integers $n \ge k \ge 2$. This result, while more general, appears less elegant. There should also be a generalization of our third formula which removes the number theoretic assumption as well, but the outcome might be an inelegant complicated formula. We do not provide such a generalization in the present work. On the other hand, our fourth formula applies for all integers $n \ge k \ge 2$.

\subsection{Related Works}
Perhaps the most closely related to our work is the recent work by De~Loera and Wesley~\cite{DLW22} (see also~\cite{DLW23, Wes23}), in which Ramsey and other combinatorial numbers were investigated using algebraic tools including polynomial ideals and Nullstellensatz methods. We remark that their results are quite general: they are applicable to {\rm multicolor graph Ramsey numbers} $R(G_1, \ldots, G_k)$ and other ``Ramsey-type'' numbers (see~\cite{DLW22} for details).

Similar to our present work, \cite{DLW22} obtains an algebraic condition which holds if and only if $n \ge R(G_1, \ldots, G_k)$. In their case, the algebraic condition is the nonexistence of solutions a system of polynomial equations in $\ol{\mathbb{F}_2}$. For the diagonal Ramsey numbers, they also have an explicit combinatorial formula depending on $n$ and $k$ (\cite{DLW22} uses $n$ and $r$) which, if true, guarantees a lower bound $R(k) > n$.

Our arguments in the present work rely on creating and studying polynomials in the commutative variables associated to the edges of a graph which encode properties we desire. While this method is quite standard, from it we are able to obtain combinatorial formulas which appear to be new, and which display connections between diagonal Ramsey numbers and other combinatorial objects. Indeed, this technique was also used by De~Loera and Wesley~\cite{DLW22}, and it is related to many important ideas in combinatorics such as the polynomial method (see e.g.~\cite{Gut13, Gut16, Tao14, Zha23}) and the combinatorial Nullstellensatz~\cite{Alo99}.

On the algebraic side of graph theory, the tool we use of associating commutative variables to graphs has also been well used. For example, Pak and Postnikov~\cite{PP94} studied spanning trees by associating variables to the vertices of a graph. The polynomial they defined turns out to have a nice reciprocity property, which was investigated further by Huang and Postnikov~\cite{HP09} and by Hoey and Xiao~\cite{HX19}.

Finally, our current work is reminiscent of the famous work of Lagarias' \cite{Lag02}, in which an elementary problem in number theory is presented. Lagarias' Problem E---although {\em elementary} in the sense that it involves simply the sum-of-divisors function $\sigma(n)$, the harmonic number $H_n$, the exponentiation function, and the logarithm function---is equivalent to the Riemann Hypothesis. In our work, we present formulas which involve ``elementary'' trigonometric functions and combinatorial enumerations, and they turn out to encode the diagonal Ramsey numbers $R(k)$.

\subsection{Outline}
The rest of this paper is organized as follows. In Section~\ref{sec:defns-main-results}, we give some definitions and describe our main results. In Section~\ref{sec:general-algebraic}, we prove some algebraic lemmas which we use later in the paper. Section~\ref{s:proofs-cos-thms} gives a proof Theorem~\ref{thm:cos-alpha-beta}, which generalizes our first and second formulas. In Section~\ref{sec:third-formula}, we establish our third formula. In Section~\ref{s:proof-Q-n-k}, we prove the fourth formula. Section~\ref{sec:ideas-for-further-investigation} lists some ideas for further investigation.

\bigskip

\section{Formal Definitions and Main Results}\label{sec:defns-main-results}
\subsection{Preliminaries}
We use the usual notation $[n] := \{1, 2, \ldots, n\}$. If $S$ is a set and $k \ge 0$ is an integer, then $\binom{S}{k}$ denotes the set of all $k$-element subsets of $S$. When $S$ is a finite set, we use both $|S|$ and $\#S$ to denote its cardinality.

Throughout this paper, graphs are finite, simple, and undirected. For each graph $G$, we let ${\bf V}(G)$ and ${\bf E}(G)$ denote the vertex set and the edge set of $G$, respectively. For $n \in \mathbb{Z}_{\ge 1}$ and $p \in [0,1]$, we let $\bbG(n,p)$ denote the Erd\H{o}s--R\'{e}nyi random graph model with edge-probability parameter $p$ on the vertex set $[n]$. When $S$ is a finite set of indices, we let $K_S$ denote the complete graph, and $\ol{K_S}$ the empty graph, on the vertex set $S$. If $G$ is a graph, we let $\ol{G}$ denote the \defn{complement} of $G$: the graph $\ol{G}$ with ${\bf V}(\ol{G}) = {\bf V}(G)$ and ${\bf E}(\ol{G}) = \binom{{\bf V}(G)}{2} - {\bf E}(G)$.

For each positive integer $k \ge 2$, the \defn{$k^{\text{th}}$ diagonal Ramsey number} $R(k)$ is defined to be the least positive integer $n$ such that every graph $G$ on $n$ vertices contains either a clique of size $k$ or an independent set of size $k$ (see e.g.~\cite{Erd90, Rad21}). A graph $G$ is said to be \defn{$k$-Ramsey} (see e.g.~\cite{Alo90, DLW22, Sah23}), or simply \defn{Ramsey} when $k$ is understood, if it does not contain a clique of size $k$ and does not contain an independent set of size $k$. We use the terms ``independent set'' and ``anticlique'' interchangeably.

For each positive integer $k \ge 2$, a \defn{finite $k$-uniform hypergraph} $H$ is an ordered pair $(V,E)$, where the \defn{vertex set} $V = {\bf V}(H)$ is a finite set of vertices, and the \defn{edge set} $E = {\bf E}(H)$ is a subset of $\binom{[n]}{k}$. Throughout this paper, we only consider finite hypergraphs, and so we may drop the adjective ``finite'' when describing a hypergraph.

For any polynomial $P(t) \in \mathbb{C}[t]$ and for any nonnegative integer $m$, we use $[t^m] \, P(t) \in \mathbb{C}$ to denote the coefficient of the term $t^m$ in $P(t)$.

\subsection{Main Results}
We now describe our main results of the paper. Throughout this subsection, let us fix positive integers $n \ge k \ge 2$.

\subsubsection{First Formula}
For each subset $U \subseteq \binom{[n]}{k}$ and for each edge $e \in \binom{[n]}{2}$, we define the \defn{multiplicity} of $e$ in $U$ as
\[
{\tt mult}(e,U) := \# \left\{ S \in U \, : e \subseteq S \right\}.
\]
We have the following formula, which encodes diagonal Ramsey numbers in terms of an expression involving cosine functions.

\begin{theorem}\label{thm:cos-mult}
For positive integers $n \ge k \ge 3$ such that $\frac{(n-2)!}{k!(n-k)!} \in \mathbb{Z}$, we have
\[
\sum_{U \subseteq \binom{[n]}{k}} \prod_{e \in \binom{[n]}{2}} \cos\!\left( \frac{2\pi}{k(k-1)} \cdot {\tt mult}(e,U) \right) = 0
\]
if and only if $n \ge R(k)$.
\end{theorem}

\smallskip

\subsubsection{Second Formula}
Our second formula follows from our first formula above together with some trigonometric manipulations. To describe the second formula, we define some additional notation.

For a finite simple undirected graph $G$ and a $k$-uniform hypergraph $H$ both on the vertex set $[n]$, define the \defn{incidence number} of $G$ and $H$ as
\[
{\tt i}(G,H) := \#\left\{ (u,v) \in {\bf E}(G) \times {\bf E}(H) \, : \, u \subseteq v \right\}.
\]

Our second formula is another trigonometric formula for diagonal Ramsey numbers, under the same divisibility assumption as in Theorem~\ref{thm:cos-mult} above.

\begin{theorem}\label{thm:cos-incidence}
For positive integers $n \ge k \ge 3$ such that $\frac{(n-2)!}{k!(n-k)!} \in \mathbb{Z}$, we have
\[
\sum_{G, H} (-1)^{|{\bf E}(H)|} \cos \!\left( \frac{4\pi}{k(k-1)} \cdot {\tt i}(G,H) \right) = 0
\]
if and only if $n \ge R(k)$. Here, the summation is over all graphs $G$ and all $k$-uniform hypergraphs $H$ on the vertex set $[n]$.
\end{theorem}

We remark that for each $k \ge 3$, there are always infinitely many positive integers $n \ge k$ for which the divisibility assumption $\frac{(n-2)!}{k!(n-k)!} \in \mathbb{Z}$ holds. For instance, when $k = 5$, the fraction is an integer whenever $n$ (is at least $k = 5$ and) is congruent to one of $15$ different residue classes modulo $40$. In particular, $n = 43, 44, 48$ all satisfy this assumption for $k = 5$. Note that, at the time of writing, the best known bounds for $R(5)$ are $43 \le R(5) \le 48$ (see \cite{Exo89, AM18}; see also Radziszowski's dynamic survey \cite{Rad21}).

On the other hand, we also have a more general (but also more complicated) version of Theorems~\ref{thm:cos-mult} and~\ref{thm:cos-incidence}. In Section~\ref{s:proofs-cos-thms} below, we present a general result (Theorem~\ref{thm:cos-alpha-beta}) which does not require the divisibility assumption.

\smallskip

\subsubsection{Third Formula}\label{subsubsection: P-n-k-formula}
Let $\calA$ denote the set of assignments $\fg$ which associate to each $k$-element subset
\[
S \in \binom{[n]}{k}
\]
a finite simple undirected graph $\fg_S$ on the vertex set $S$ which is neither $K_S$ nor $\ol{K_S}$.

For each $\fg \in \calA$, we define the \defn{set of edges} of $\fg$ as
\[
{\tt Edges}(\fg) := \bigcup_{S \in \binom{[n]}{k}} {\bf E}(\fg_S) \subseteq \binom{[n]}{2},
\]
and define the \defn{sign} of $\fg$ as
\[
{\tt sign}(\fg) := \prod_{S \in \binom{[n]}{k}} (-1)^{|{\bf E}(\fg_S)|} \in \{-1,+1\}.
\]
The following polynomial is a signed enumeration of $\calA$ which keeps track of the cardinality of the set of edges:
\begin{equation}\label{eq:P-n-k-t-defn}
P_{n,k}(t) := \sum_{\fg \in \calA} {\tt sign}(\fg) \cdot t^{|{\tt Edges}(\fg)|} \in \mathbb{Z}[t].
\end{equation}

Observe that we may express $P_{n,k}(t)$ as
\[
P_{n,k}(t) = \sum_{r \ge 0} \left(A^+(r) - A^-(r) \right) t^r \in \mathbb{Z}[t],
\]
where, for each $r \in \mathbb{Z}_{\ge 0}$,
\[
A^+(r) := \#\left\{ \fg \in \calA \, | \, {\tt sign}(\fg) = +1 \, \text{ and } |{\tt Edges}(\fg)| = r \right\},
\]
and
\[
A^-(r) := \#\left\{ \fg \in \calA \, | \, {\tt sign}(\fg) = -1 \, \text{ and } |{\tt Edges}(\fg)| = r \right\}.
\]

The following result says that the univariate polynomials $P_{n,k}(t)$, when $k \equiv 2, 3 \pmod{4}$, ``encode'' the diagonal Ramsey number $R(k)$.

\begin{theorem}\label{thm:P-n-k}
Let $n \ge k \ge 2$ be positive integers such that $k$ is congruent to $2$ or $3$ modulo~$4$. Then
\[
\bbP \big( G \sim \bbG(n,1/2) \text{ is a } k\text{-Ramsey graph} \big) = (-1)^{\binom{n}{k}} \cdot P_{n,k}\!\left( \frac{1}{2} \right).
\]
\end{theorem}

We prove Theorem~\ref{thm:P-n-k} in Subsection~\ref{subsection:proof-thm-P-n-k}. From the theorem, we have the following corollary.

\begin{corollary}\label{cor:P-nk-1-2}
Let $n$ and $k$ be as in the above theorem. The following are equivalent.
\begin{itemize}
    \item[(i)] $P_{n,k}$ vanishes at $1/2$.
    \item[(ii)] $P_{n,k}$ is the zero polynomial.
    \item[(iii)] $n \ge R(k)$.
\end{itemize}
\end{corollary}

We remark that as a result of Corollary~\ref{cor:P-nk-1-2}, for $n, k$ as in the corollary, we only need to check that $P_{n,k}(1/2) = 0$ to obtain $n \ge R(k)$. On the other hand, if we know $n \ge R(k)$, then we have that $P_{n,k}$ is zero as a polynomial, and so it vanishes everywhere not just at $1/2$. We prove Corollary~\ref{cor:P-nk-1-2} in Subsection~\ref{subsection:proof-cor-P-n-k}.

We discuss further properties of the polynomials $P_{n,k}(t)$ in Subsection~\ref{subsection:remarks-P-n-k}.

\smallskip

\subsubsection{Fourth Formula}\label{subsubsection: Q-n-k-formula}
Let $\calB$ denote the set of assignments $\fh$ which associate to each $k$-element subset
\[
S \in \binom{[n]}{k}
\]
a finite simple undirected graph $\fh_S$ on the vertex set $S$ such that
\begin{itemize}
    \item[(i)] for every $S \in \binom{[n]}{k}$, the size $|{\bf E}(\fh_S)|$ is an even nonnegative integer, and
    \item[(ii)] for every edge $e \in \binom{[n]}{2}$, the number of $S \in \binom{[n]}{k}$ for which $e \in {\bf E}(\fh_S)$ is an even nonnegative integer.
\end{itemize}

For each $\fh \in \calB$, define the \defn{empty locus} of $\fh$ to be
\begin{equation}\label{eq:tt-Empty-fh}
{\tt Empty}(\fh) := \left\{ S \in \binom{[n]}{k} \, : \, \fh_S \text{ is an empty graph} \right\}.
\end{equation}

The following polynomial is an enumeration of $\calB$ which keeps track of the cardinality of the empty locus:
\begin{equation}\label{eq:Q-n-k-t-defn}
Q_{n,k}(t) := \sum_{\fh \in \calB} t^{|{\tt Empty}(\fh)|} \in \mathbb{Z}[t].
\end{equation}

The following result says that the polynomials $Q_{n,k}$ also ``encode'' the diagonal Ramsey numbers $R(k)$. Note that, unlike the three formulas above, this theorem does not have a number theoretic restriction on $n$ and $k$. Hence, it applies to all integers $n \ge k \ge 2$.

\begin{theorem}\label{thm:Q-n-k}
Let $n \ge k \ge 2$ be positive integers. Then
\[
\bbP \big( G \sim \bbG(n,1/2) \text{ is a } k\text{-Ramsey graph} \big) = (-1)^{\binom{n}{k}} 2^{\left( 1 - \binom{k}{2} \right)\binom{n}{k}} \cdot Q_{n,k}\!\left( 1 - 2^{\binom{k}{2} - 1} \right).
\]
\end{theorem}

We have the following immediate corollary.

\begin{corollary}\label{cor:Q-n-k}
Let $n \ge k \ge 2$ be integers. Then $1 - 2^{\binom{k}{2} - 1}$ is a root of $Q_{n,k}$ if and only if $n \ge R(k)$.
\end{corollary}

We prove Theorem~\ref{thm:Q-n-k} in Section~\ref{s:proof-Q-n-k}.

\bigskip

\section{General Algebraic Framework}\label{sec:general-algebraic}
In this section, we give a general framework for our arguments. Throughout this section, we let $\calR$ be a commutative ring with additive identity $0$ and multiplicative identity $1$. Let $P \subseteq \calR$ be a subset with the following properties:
\begin{itemize}
    \item $0 \notin P$,
    \item if $x,y \in P$, then $x+y \in P$, and
    \item if $x,y \in P$, then $xy \in P$.
\end{itemize}

Let $k \ge 2$ be a positive integer, and we denote by $\mathcal{A}_k$ the set of $\{0,1\}$-square $k \times k$ symmetric hollow matrices. (A square matrix is said to be \defn{hollow} if all its diagonal entries are zero.) In other words, $\mathcal{A}_k$ is the set of all possible adjacency matrices of graphs on $k$ vertices. Let $f: \mathcal{A}_k \to \calR$ be a function with the following properties.
\begin{itemize}
    \item if $M$ is the adjacency matrix of a clique or an empty graph, then $f(M) = 0$, and
    \item if $M$ is not, then $f(M) \in P$.
\end{itemize}

\begin{lemma}\label{l:sum-prod-f-A-S}
Let $n \ge k \ge 2$ be positive integers. Let $f: \mathcal{A}_k \to \calR$ be as above. For each $S \in \binom{[n]}{k}$ and for each $n \times n$ matrix $A$, let us write $A_S$ for the $k \times k$ submatrix of $A$ with row and column indices from $S \subseteq [n]$ (with row and column orderings induced from $[n]$). Then
\begin{equation}\label{eq:sum-prod-f-A-S}
\sum_{A \in \mathcal{A}_n} \prod_{S \in \binom{[n]}{k}} f(A_S) = 0
\end{equation}
if and only if $n \ge R(k)$.
\end{lemma}

\begin{proof}
If $n \ge R(k)$, then for every graph $G$ on the vertex set $[n]$, there exists $S \in \binom{[n]}{k}$ such that the induced subgraph of $G$ on the vertex set $S$ is either a clique or an anticlique. Thus for any $A \in \mathcal{A}_n$, there is $S \in \binom{[n]}{k}$ such that $f(A_S) = 0$. Hence the product in~\eqref{eq:sum-prod-f-A-S} is zero, for every $A \in \mathcal{A}_n$.

If $n < R(k)$, there exists at least one graph $G$ on the vertex set $[n]$ without any $k$-clique or $k$-anticlique. In this case, the product in~\eqref{eq:sum-prod-f-A-S} is an element in $P$. Hence, the sum in~\eqref{eq:sum-prod-f-A-S} is a sum of a positive integer number of elements in $P$, and zeros. Thus, the sum is nonzero.
\end{proof}

For the next lemma, we let $I$ be a finite index set. Consider the set
\[
\calX := \{X_i \, : \, i \in I\}
\]
of commutative variables, and the polynomial ring $\calR[\calX]$. (If $I = \varnothing$, then this ring is simply $\calR$ itself.) Let $\{0,1\}^I$ denote the set of assignments ${\bf x} \in \{0,1\}^I$ which associate
\[
{\bf x}: i \mapsto x_i \in \{0,1\}
\]
for every $i \in I$.

\begin{lemma}\label{l:F-sum-a-U-prod-zero}
Let $a_\calU \in \calR$, for each subset $\calU \subseteq I$. Consider the polynomial
\[
F = \sum_{\calU \subseteq I} a_\calU \prod_{i \in U} X_i \in \calR[\calX].
\]
Suppose that for every ${\bf x} \in \{0,1\}^I$,
\[
F\big|_{\forall i \in I, X_i = x_i} = 0.
\]
Then $F$ is the zero polynomial.
\end{lemma}
\begin{proof}
We induct on $|I|$. The statement is clear when $I = \varnothing$. Now suppose $I \neq \varnothing$. Take an arbitrary element $\iota \in I$. For each $\calU \subseteq I$, let us use the shorthand
\[
X_\calU := \prod_{i \in \calU} X_i.
\]

Suppose
\[
F = \sum_{\calU \subseteq I} a_\calU X_\calU \in \calR[\calX]
\]
is a polynomial that vanishes at every point ${\bf x} \in \{0,1\}^I$.

Write
\[
F = \sum_{\calU \subseteq I - \{\iota\}} a_\calU X_\calU + \sum_{\calU \subseteq I - \{\iota\}} a_{\calU \cup \{\iota\}} X_\calU X_\iota.
\]
By considering when $X_{\iota} = 0$ and when $X_{\iota} = 1$, we find that both polynomials
\[
\sum_{\calU \subseteq I - \{\iota\}} a_\calU X_\calU
\]
and
\[
\sum_{\calU \subseteq I - \{\iota\}} a_{\calU \cup \{\iota\}} X_\calU
\]
vanish at every point in $\{0,1\}^{I - \{\iota\}}$. By the inductive hypothesis, we conclude that $a_\calU = a_{\calU \cup \{\iota\}} = 0$ for every $\calU \subseteq I - \{\iota\}$.
\end{proof}

\bigskip

\section{The first and second formulas}\label{s:proofs-cos-thms}
The goal of this section is to prove the following theorem.

\begin{theorem}\label{thm:cos-alpha-beta}
Let $n \ge k \ge 2$ be positive integers. The following are equivalent.
\begin{itemize}
    \item[(i)] $n \ge R(k)$.
    \item[(ii)] For any real number $q \in \mathbb{R}$ and for any integer $m \in \mathbb{Z}$, we have
    \[
    \sum_{U \subseteq \binom{[n]}{k}} (-1)^{(m + 1)|U|} \prod_{e \in \binom{[n]}{2}} \cos\!\left( \frac{2\pi}{k(k-1)} \left( q \binom{n-2}{k-2} + m \cdot {\tt mult}(e,U) \right) \right) = 0.
    \]
    \item[(iii)] For any real number $q \in \mathbb{R}$ and for any integer $m \in \mathbb{Z}$, we have
    \[
    \sum_{G,H} (-1)^{|{\bf E}(H)|} \cos\!\left( \frac{\pi q (n-2)!}{k!(n-k)!} \cdot \Big(4|{\bf E}(G)| - n(n-1)\Big) + \frac{4\pi m}{k(k-1)} \cdot {\tt i}(G,H) \right) = 0,
    \]
    where the summation is over all graphs $G$ and all $k$-uniform hypergraphs $H$ on the vertex set $[n]$.
    \item[(iv)] The following equality holds.
    \[
    \sum_{U \subseteq \binom{[n]}{k}} \prod_{e \in \binom{[n]}{2}} \cos\!\left( - \frac{\pi \cdot (n-2)!}{k! (n-k)!} + \frac{2\pi}{k(k-1)} \cdot {\tt mult}(e,U) \right) = 0.
    \]
    \item[(v)] The following equality holds.
    \[
    \sum_{G,H} (-1)^{|{\bf E}(H)|} \cos\!\left( - \frac{\pi (n-2)!}{k! (n-k)!} \left( 2|{\bf E}(G)| - \binom{n}{2}\right) + \frac{4\pi}{k(k-1)} \cdot {\tt i}(G,H) \right) = 0,
    \]
    where the summation is over all graphs $G$ and all $k$-uniform hypergraphs $H$ on the vertex set $[n]$.
\end{itemize}
\end{theorem}

Before proving Theorem~\ref{thm:cos-alpha-beta}, let us quickly note how to obtain Theorems~\ref{thm:cos-mult} and \ref{thm:cos-incidence} from Theorem~\ref{thm:cos-alpha-beta}.

\begin{proof}[Proof of Theorems~\ref{thm:cos-mult} and \ref{thm:cos-incidence}]
Suppose that
\[
a := \frac{(n-2)!}{k!(n-k)!} \in \mathbb{Z}.
\]
Then the expression on the left-hand side of~(iv) in Theorem~\ref{thm:cos-alpha-beta} can be written as
\begin{align*}
&\sum_{U \subseteq \binom{[n]}{k}} \prod_{e \in \binom{[n]}{2}} \cos\!\left( - \frac{\pi \cdot (n-2)!}{k! (n-k)!} + \frac{2\pi}{k(k-1)} \cdot {\tt mult}(e,U) \right) \\
&= (-1)^{a \cdot n(n-1)/2} \sum_{U \subseteq \binom{[n]}{k}} \prod_{e \in \binom{[n]}{2}} \cos\!\left( \frac{2\pi}{k(k-1)} \cdot {\tt mult}(e,U) \right).
\end{align*}
Hence, Theorem~\ref{thm:cos-mult} follows from (i) $\Leftrightarrow$ (iv).

\smallskip

Similarly, the expression on the left-hand side of~(v) in Theorem~\ref{thm:cos-alpha-beta} can be written as
\begin{align*}
& \sum_{G,H} (-1)^{|{\bf E}(H)|} \cos\!\left( - \frac{\pi (n-2)!}{k! (n-k)!} \left( 2|{\bf E}(G)| - \binom{n}{2}\right) + \frac{4\pi}{k(k-1)} \cdot {\tt i}(G,H) \right) \\
&= (-1)^{a \cdot n(n-1)/2} \sum_{G,H} (-1)^{|{\bf E}(H)|} \cos\!\left( \frac{4\pi}{k(k-1)} \cdot {\tt i}(G,H) \right).
\end{align*}
Hence, Theorem~\ref{thm:cos-incidence} follows from (i) $\Leftrightarrow$ (v).
\end{proof}

Now we prove Theorem~\ref{thm:cos-alpha-beta}.
\begin{proof}[Proof of Theorem~\ref{thm:cos-alpha-beta}]
{\bf (i) $\Rightarrow$ (ii).} Suppose that $n \ge R(k)$. Let $q \in \mathbb{R}$ and $m \in \mathbb{Z}$. Denote by $\{0,1\}^{\binom{[n]}{2}}$ the set of all assignments ${\bf x} \in \{0,1\}^{\binom{[n]}{2}}$ which associate
\[
{\bf x}: e \mapsto x_e \in \{0,1\},
\]
for every edge $e \in \binom{[n]}{2}$. Using the usual one-to-one correspondence between $\{0,1\}^{\binom{[n]}{2}}$ and the graphs on the vertex set $[n]$ and the assumption $n \ge R(k)$, we see that for any ${\bf x}$ there exists a $k$-element subset $S \in \binom{[n]}{k}$ such that either
\begin{itemize}
    \item $\forall e \in \binom{S}{2}, x_e = 0$, or
    \item $\forall e \in \binom{S}{2}, x_e = 1$.
\end{itemize}
For such a subset $S \in \binom{[n]}{k}$, we see that
\[
\exp\!\bigg( \frac{4\pi(q+m) i}{k(k-1)} \sum_{e \in \binom{S}{2}} x_e \bigg) - \exp\!\bigg( \frac{4\pi q i}{k(k-1)} \sum_{e \in \binom{S}{2}} x_e\bigg) = 0.
\]
Therefore, for any ${\bf x} \in \{0,1\}^{\binom{[n]}{2}}$, we obtain
\[
\prod_{S \in \binom{[n]}{k}} \left( \exp\!\bigg( \frac{4\pi(q+m) i}{k(k-1)} \sum_{e \in \binom{S}{2}} x_e \bigg) - \exp\!\bigg( \frac{4\pi q i}{k(k-1)} \sum_{e \in \binom{S}{2}} x_e\bigg) \right) = 0.
\]
Expanding the product, we find
\[
\sum_{U \subseteq \binom{[n]}{k}} \exp\!\bigg( \frac{4\pi(q+m)i}{k(k-1)} \sum_{S \in U} \sum_{e \in \binom{S}{2}} x_e \bigg)(-1)^{\binom{n}{k} - |U|} \exp\!\bigg( \frac{4\pi qi}{k(k-1)} \sum_{S \in \binom{[n]}{k} - U} \sum_{e \in \binom{S}{2}} x_e \bigg) = 0.
\]
Hence,
\begin{equation}\label{eq:sum-U-q-binom-m-mult}
\sum_{U \subseteq \binom{[n]}{k}} (-1)^{|U|} \exp\!\left( \frac{4\pi i}{k(k-1)} \sum_{e \in \binom{[n]}{2}} x_e \left( q \binom{n-2}{k-2} + m \cdot {\tt mult}(e,U) \right) \right) = 0.
\end{equation}

Recall that Equation~\eqref{eq:sum-U-q-binom-m-mult} holds for any ${\bf x} \in \{0,1\}^{\binom{[n]}{2}}$. Therefore, we can sum the left-hand side of~\eqref{eq:sum-U-q-binom-m-mult} over all ${\bf x} \in \{0,1\}^{\binom{[n]}{2}}$ to obtain
\[
\sum_{{\bf x}} \sum_{U \subseteq \binom{[n]}{k}} (-1)^{|U|} \exp\!\left( \frac{4\pi i}{k(k-1)} \sum_{e \in \binom{[n]}{2}} x_e \left( q \binom{n-2}{k-2} + m \cdot {\tt mult}(e,U) \right) \right) = 0,
\]
and thus
\[
\sum_{U \subseteq \binom{[n]}{k}} (-1)^{|U|} \prod_{e \in \binom{[n]}{2}} \left( 1 +  \exp\!\left( \frac{4\pi i}{k(k-1)} \left( q \binom{n-2}{k-2} + m \cdot {\tt mult}(e,U) \right) \right) \right) = 0.
\]
Using the identity $1+\exp(2i \theta) = 2 \cos \theta \cdot \exp(i \theta)$ and dividing through by $2^{\binom{n}{2}}$, we obtain
\begin{equation}\label{eq:sum-U-T1-T2-0}
\sum_{U \subseteq \binom{[n]}{k}} (-1)^{|U|} \cdot T_1(U) \cdot T_2(U) = 0,
\end{equation}
where
\begin{equation}\label{eq:T1-U-defn}
T_1(U) := \prod_{e \in \binom{[n]}{2}} \cos \!\left( \frac{2\pi}{k(k-1)} \left( q \binom{n-2}{k-2} + m \cdot {\tt mult}(e,U) \right) \right),
\end{equation}
and
\[
T_2(U) := \prod_{e \in \binom{[n]}{2}} \exp \!\left( \frac{2\pi i}{k(k-1)} \left( q \binom{n-2}{k-2} + m \cdot {\tt mult}(e,U) \right) \right).
\]
By noticing that
\[
\sum_{e \in \binom{[n]}{2}} {\tt mult}(e,U) = \sum_{e \in \binom{[n]}{2}} \sum_{\substack{s \in U \\ e \subseteq S}} 1 = \sum_{S \in U} \sum_{e \subseteq S} 1 = |U| \cdot \binom{k}{2},
\]
we find that
\begin{equation}\label{eq:T2-U-1-q-n-k-1-m-U}
T_2(U) = \exp\!\left( \pi i \cdot q \binom{n}{k} + \pi i \cdot m |U| \right) = (-1)^{q \binom{n}{k}} (-1)^{m |U|}.
\end{equation}
Combining~\eqref{eq:sum-U-T1-T2-0}, \eqref{eq:T1-U-defn}, and \eqref{eq:T2-U-1-q-n-k-1-m-U} yields (ii), as desired.

\medskip

{\bf (ii) $\Leftrightarrow$ (iii).} Let us define
\[
T_3 := \sum_{U \subseteq \binom{[n]}{k}} (-1)^{(m + 1)|U|} \prod_{e \in \binom{[n]}{2}} \cos\!\left( \frac{2\pi}{k(k-1)} \left( q \binom{n-2}{k-2} + m \cdot {\tt mult}(e,U) \right) \right),
\]
and
\[
T_4 := \sum_{G,H} (-1)^{|{\bf E}(H)|} \cos\!\left( \frac{\pi q (n-2)!}{k!(n-k)!} \cdot \Big(4|{\bf E}(G)| - n(n-1)\Big) + \frac{4\pi m}{k(k-1)} \cdot {\tt i}(G,H) \right).
\]
Let us denote by $\{-1,1\}^{\binom{[n]}{2}}$ the set of all assignments $\varepsilon_{\bullet} \in \{-1,1\}^{\binom{[n]}{2}}$ which associate
\[
\varepsilon_{\bullet}: e \mapsto \varepsilon_e \in \{-1,1\},
\]
for every edge $e \in \binom{[n]}{2}$. Using the multivariable generalization of the identity $\cos \theta \cos \eta = \frac{1}{2}\left( \cos(\theta + \eta) + \cos(\theta - \eta) \right)$ (which is called a {\em prosthaphaeresis} formula, among other names---see e.g.~\cite{Tho88}), we write
\begin{equation}\label{eq:T3-eps-bullet}
T_3 = \sum_{U \subseteq \binom{[n]}{k}} (-1)^{(m+1)|U|} \cdot 2^{-\binom{n}{2}} \sum_{\varepsilon_\bullet} \cos\!\left( \frac{2\pi}{k(k-1)} \cdot A_1(U,\varepsilon_\bullet) \right),
\end{equation}
where the inner summation is over all $\varepsilon_\bullet \in \{-1,1\}^{\binom{[n]}{2}}$, and
\begin{equation}\label{eq:A1-U-eps-defn}
A_1(U,\varepsilon_\bullet) := \sum_{e \in \binom{[n]}{2}} \varepsilon_e \left( q \binom{n-2}{k-2} + m \cdot {\tt mult}(e,U) \right).
\end{equation}

By writing
\[
W = W(\varepsilon_\bullet) := \left\{ e \in \binom{[n]}{2} \, : \, \varepsilon_e = +1 \right\},
\]
we express
\begin{align}
A_1(U, \varepsilon_\bullet) &= 2 \sum_{e \in W} \left( q \binom{n-2}{k-2} + m \cdot {\tt mult}(e,U) \right) - \sum_{e \in \binom{[n]}{2}} \left( q \binom{n-2}{k-2} + m \cdot {\tt mult}(e,U) \right) \notag \\
&= 2q \binom{n-2}{k-2} |W| + 2m \cdot {\tt mult}(W,U) - q \binom{n-2}{k-2} \binom{n}{2} - m |U| \binom{k}{2}, \label{eq:A1-U-eps-2q-binom}
\end{align}
where
\[
{\tt mult}(W,U) := \sum_{e \in W} {\tt mult}(e,U).
\]

Combining \eqref{eq:T3-eps-bullet}, \eqref{eq:A1-U-eps-defn}, and \eqref{eq:A1-U-eps-2q-binom}, we find
\begin{align*}
&2^{\binom{n}{2}} T_3 \\
&= \sum_{U \subseteq \binom{[n]}{k}} \sum_{W \subseteq \binom{[n]}{2}} (-1)^{|U|} \cos\!\left( \frac{\pi q (n-2)!}{k! (n-k)!} (4|W| - n(n-1)) + \frac{4\pi m}{k(k-1)} \cdot {\tt mult}(W,U) \right).
\end{align*}
By thinking of $U$ and $W$ as the sets of edges of a $k$-uniform hypergraph $H$ and a graph $G$ on the vertex set $[n]$, respectively, and noting that
\[
{\tt mult}(W,U) = {\tt i}(G,H),
\]
we obtain
\[
2^{\binom{n}{2}} T_3 = T_4.
\]
Hence, the equivalence between (ii) and (iii) follows.

\medskip

{\bf (ii) $\Rightarrow$ (iv).} This follows by simply plugging in $q = -1/2$ and $m = 1$.

\medskip

{\bf (iv) $\Leftrightarrow$ (v).} We have shown above that $2^{\binom{n}{2}} T_3 = T_4$, and hence
\[
2^{\binom{n}{2}} T_3\Big|_{(q,m) = (-1/2,1)} = T_4\Big|_{(q,m) = (-1/2,1)}.
\]

{\bf (v) $\Rightarrow$ (i).} With the same trigonometric manipulations as before, we find that if
\[
T_4\Big|_{(q,m) = (-1/2,1)} = 0,
\]
then
\[
\sum_{{\bf x}} \prod_{S \in \binom{[n]}{k}} \left( \exp\!\bigg( \frac{4\pi(q+m) i}{k(k-1)} \sum_{e \in \binom{S}{2}} x_e \bigg) - \exp\!\bigg( \frac{4\pi q i}{k(k-1)} \sum_{e \in \binom{S}{2}} x_e\bigg) \right) \Bigg|_{(q,m) = (-1/2,1)} = 0,
\]
which means
\[
\sum_{{\bf x}} \prod_{S \in \binom{[n]}{k}} \sin\!\left( \frac{\pi}{\binom{k}{2}} \sum_{e \in \binom{S}{2}} x_e \right) = 0.
\]

We finish by invoking Lemma~\ref{l:sum-prod-f-A-S} with the function
\[
f\!\left( {\bf x}|_S \right) = \sin\!\left( \frac{\pi}{\binom{k}{2}} \sum_{e \in \binom{S}{2}} x_e \right),
\]
the ring $\calR = \mathbb{C}$, and the subset $P = \mathbb{R}_{> 0}$. Here, we identify the set $\mathcal{A}_n$ of adjacency matrices with $\{0,1\}^{\binom{[n]}{2}}$, and the notation ${\bf x}|_S$ refers to the restriction of ${\bf x}$ to $\binom{S}{2}$, which we identify with a $k \times k$ adjacency matrix.
\end{proof}

\bigskip

\section{The third formula}\label{sec:third-formula}
In this section, we prove Theorem~\ref{thm:P-n-k} and Corollary~\ref{cor:P-nk-1-2}. Throughout this section, we fix positive integers $n \ge k \ge 2$ such that $k \equiv 2 \text{ or } 3 \pmod{4}$.

\subsection{Proof of Theorem~\ref{thm:P-n-k}}\label{subsection:proof-thm-P-n-k}
As in the previous section, we continue to think of a graph on the vertex set $[n]$ as an assignment ${\bf x} \in \{0,1\}^{\binom{[n]}{2}}$, which associates
\[
{\bf x}: e \mapsto x_e \in \{0,1\}
\]
to every edge $e \in \binom{[n]}{2}$. As we did earlier, for each $S \in \binom{[n]}{k}$, let us denote by ${\bf x}|_S$ the restriction of ${\bf x}$ to $\binom{S}{2}$.

The main idea of the proof is to apply Lemma~\ref{l:sum-prod-f-A-S} with the function
\begin{equation}\label{eq:f-1-prod-prod}
f\!\left( {\bf x}|_S \right) = 1 - \prod_{e \in \binom{S}{2}} x_e - \prod_{e \in \binom{S}{2}} (1-x_e),
\end{equation}
the ring $\calR = \mathbb{Z}$, and the subset $P = \mathbb{Z}_{\ge 1}$.

We make the following observation about plugging the function $f$ from~\eqref{eq:f-1-prod-prod} into the equation in Lemma~\ref{l:sum-prod-f-A-S}.

\begin{observation}\label{obs:enum-form-x-0-1}
While Lemma~\ref{l:sum-prod-f-A-S} implies that
\[
\sum_{{\bf x}} \prod_{S \in \binom{[n]}{k}} f({\bf x}|_S) = 0
\]
if and only if $n \ge R(k)$, we can say something more about this function. Suppose that ${\bf x} \in \{0,1\}^{\binom{[n]}{2}}$ is given. This means we have a graph $G$ on the vertex set $[n]$. For each $S \in \binom{[n]}{k}$, let us denote by $G_S$ the induced subgraph of $G$ on the vertex set $S$.

We note that $f({\bf x}|_S)$ equals $0$ if and only if $G_S$ is either a clique or an anticlique; otherwise, $f({\bf x}|_S)$ equals $1$. Thus, the product
\[
\prod_{S \in \binom{[n]}{k}} f({\bf x}|_S)
\]
equals $1$ if and only if $G$ is a $k$-Ramsey graph; otherwise, the product equals zero.

Therefore, the summation
\[
\sum_{{\bf x}} \prod_{S \in \binom{[n]}{k}} f({\bf x}|_S),
\]
over all ${\bf x} \in \{0,1\}^{\binom{[n]}{2}}$, is an {\em enumeration} formula for the number of $k$-Ramsey graphs on the vertex set $[n]$.
\end{observation}

The assumption $k \equiv 2 \text{ or } 3 \pmod{4}$ implies that the coefficient of the highest-degree monomial $\prod_{e \in \binom{S}{2}} x_e =: x_{\binom{S}{2}}$ in the product
\[
\prod_{e \in \binom{S}{2}} (1 - x_e)
\]
is $-1$. This means there is cancellation of the term $x_{\binom{S}{2}}$ in the expression on the right-hand side of Equation~\eqref{eq:f-1-prod-prod}. More precisely,
\begin{equation}\label{eq:f-x-S-sum-U-1-U}
f({\bf x}|_S) = \sum_{\varnothing \subsetneq \calU \subsetneq \binom{S}{2}} (-1)^{|\calU|+1} x_{\calU},
\end{equation}
where, for each subset $T \subseteq \binom{[n]}{2}$, we define
\[
x_T := \prod_{e \in T} x_e.
\]

Observation~\ref{obs:enum-form-x-0-1} then implies that
\begin{equation}\label{eq:Ramsey-sum-x-calU}
\#\left\{ k\text{-Ramsey graphs on } [n] \right\} = \sum_{{\bf x}} \prod_{S \in \binom{[n]}{k}} \sum_{\varnothing \subsetneq \calU \subsetneq \binom{S}{2}} (-1)^{|\calU|+1} x_{\calU}.
\end{equation}

Now we recall the definition of the set $\mathcal{A}$ from Subsubsection~\ref{subsubsection: P-n-k-formula}. Expanding the product on the right-hand side of Equation~\eqref{eq:Ramsey-sum-x-calU}, together with noting that $x_e^2 = x_e$ for every $e \in \binom{[n]}{2}$, gives us
\begin{align*}
\#\left\{ k\text{-Ramsey graphs on } [n] \right\} &= \sum_{{\bf x}} (-1)^{\binom{n}{k}} \sum_{\fg \in \calA} {\tt sign}(\fg) \cdot x_{{\tt Edges}(\fg)} \\
&= (-1)^{\binom{n}{k}} \sum_{\fg \in \calA} {\tt sign}(\fg) \sum_{{\bf x}} x_{{\tt Edges}(\fg)}.
\end{align*}
For each subset $T \subseteq \binom{[n]}{2}$, we note that $x_T$ equals $1$ if and only if for every edge $e \in T$, $x_e = 1$; otherwise, $x_T$ equals $0$. This shows that
\[
\sum_{{\bf x}} x_T = 2^{\binom{n}{2} - |T|}.
\]
Using this in the expression above, we find
\[
\#\left\{ k\text{-Ramsey graphs on } [n] \right\} = (-1)^{\binom{n}{k}} 2^{\binom{n}{2}} \sum_{\fg \in \calA} {\tt sign}(\fg) \left( \frac{1}{2} \right)^{|{\tt Edges}(\fg)|},
\]
from which it follows that the probability that a uniformly random graph on $[n]$ is $k$-Ramsey is $(-1)^{\binom{n}{k}} \cdot P_{n,k}(1/2)$, as desired.

\subsection{Proof of Corollary~\ref{cor:P-nk-1-2}}\label{subsection:proof-cor-P-n-k}
Let us now show that the three items in the corollary are equivalent.

\medskip

{\bf (i) $\Rightarrow$ (iii).} From Theorem~\ref{thm:P-n-k}, if $P_{n,k}$ vanishes at $1/2$, then
\[
\bbP \big( G \sim \bbG(n,1/2) \text{ is a } k\text{-Ramsey graph} \big) = 0,
\]
whence $n \ge R(k)$.

\medskip

{\bf (iii) $\Rightarrow$ (ii).} Now suppose $n \ge R(k)$. By Observation~\ref{obs:enum-form-x-0-1}, we find that the polynomial
\[
F({\bf x}) := \sum_{\fg \in \calA} {\tt sign}(\fg) \cdot x_{{\tt Edges}(\fg)}
\]
vanishes at every point ${\bf x} \in \{0,1\}^{\binom{[n]}{2}}$. Let us write
\[
F({\bf x}) = \sum_{\calU \subseteq \binom{[n]}{2}} a_\calU \cdot x_\calU,
\]
where
\[
a_\calU := \sum_{\substack{\fg \in \calA \\ {\tt Edges}(\fg) = \calU}} {\tt sign}(\fg).
\]
Lemma~\ref{l:F-sum-a-U-prod-zero} implies that $a_\calU = 0$ for every $\calU \subseteq \binom{[n]}{2}$.

Now observe that the polynomial $P_{n,k}(t) \in \mathbb{Z}[t]$ can be expressed as
\[
P_{n,k}(t) = \sum_{d \ge 0} \sum_{\substack{\fg \in \calA \\ |{\tt Edges}(\fg)| = d}} {\tt sign}(\fg) \cdot t^d = \sum_{d \ge 0} t^d \sum_{\calU \, : \, |\calU| = d} a_\calU.
\]
Therefore, $P_{n,k}$ is the zero polynomial.

\medskip

{\bf (ii) $\Rightarrow$ (i).} This implication is clear. We have finished the proof.

\subsection{Some remarks}\label{subsection:remarks-P-n-k}
We give some remarks about the polynomials $P_{n,k}(t) \in \mathbb{Z}[t]$.

\subsubsection{Matrix Formulation}\label{subsubsection: P-n-k-matrix}
It might be convenient to think about the polynomials $P_{n,k}(t)$ in terms of calculations with $\{0,1\}$-matrices. We start with the \defn{incidence matrix}
\begin{equation}\label{in:I-n-k}
{\tt I}(n,k) \in \{0,1\}^{\binom{[n]}{2} \times \binom{[n]}{k}},
\end{equation}
which is a $\{0,1\}$-matrix whose rows are indexed by $\binom{[n]}{2}$ and whose columns are indexed by $\binom{[n]}{k}$, where the $(e,S)$-entry of ${\tt I}(n,k)$ is $1$ if and only if $e \subseteq S$. The incidence matrix ${\tt I}(n,k)$ has the following properties:
\begin{itemize}
    \item its dimension is $\binom{n}{2} \times \binom{n}{k}$,
    \item each row has exactly $\binom{n-2}{k-2}$ ones, and
    \item each column has exactly $\binom{k}{2}$ ones.
\end{itemize}

Now for any $\{0,1\}$-matrices $A$ and $B$ of the same dimension, let us say $A \preceq B$ if for every $i,j$, we have $A_{ij} \le B_{ij}$. In particular, $\preceq$ defines a poset structure on the family $\{0,1\}^{\binom{[n]}{2} \times \binom{[n]}{k}}$ of $\{0,1\}$-matrices.

Let us define $\calA'$ to be the set of matrices
\[
M \in \{0,1\}^{\binom{[n]}{2} \times \binom{[n]}{k}}
\]
such that
\begin{itemize}
    \item[(i)] $M \preceq {\tt I}(n,k)$,
    \item[(ii)] every column of $M$ has at least $1$ one, and
    \item[(iii)] every column of $M$ has at most $\binom{k}{2} - 1$ ones.
\end{itemize}

For each $e \in \binom{[n]}{2}$, let $|M_{e, -}|$ denote the sum of all entries on row $e$. Furthermore, let $|M|$ denote the sum of all entries in $M$.

If we define
\[
{\tt Edges}(M) := \left\{ e \in \binom{[n]}{2} \, : \, |M_{e, -}| \ge 1 \right\},
\]
and
\[
{\tt sign}(M) := (-1)^{|M|},
\]
then it follows that
\begin{equation}\label{eq:P-n-k-t-edges-M}
P_{n,k}(t) = \sum_{M \in \calA'} {\tt sign}(M) \cdot t^{|{\tt Edges}(M)|}.
\end{equation}
This is simply a consequence of the one-to-one correspondence between $\calA$ and $\calA'$, under which $\fg \in \calA$ corresponds to a matrix $M \in \calA'$ which records the edges of graphs of $\fg$. (The rule is that $M_{e,S} = 1$ if and only if $e \in {\bf E}(\fg_S)$.) Note that under this correspondence, ${\tt Edges}(M) = {\tt Edges}(\fg)$ and ${\tt sign}(M) = {\tt sign}(\fg)$.

\subsubsection{Sign-reversing Involutions}\label{subsubsection: sign-reversing}
Throughout this subsubsection, let us fix positive integers $n \ge k \ge 2$. For each positive integer $m$, let $\calE_m \subseteq \calA'$ denote the set of all matrices $M \in \calA'$ with $|{\tt Edges}(M)| = m$. We note, from~\eqref{eq:P-n-k-t-edges-M}, that the $t^m$-coefficient of $P_{n,k}(t)$ is
\begin{equation}\label{eq:t-m-P-n-k-t-in-Z}
[t^m] \, P_{n,k}(t) = \sum_{\substack{M \in \calE_m}} (-1)^{|M|} \in \mathbb{Z}.
\end{equation}
While we do not investigate them deeply within the current paper, there seem to be many interesting {\em sign-reversing involutions}
\[
\sigma: E' \to E',
\]
defined on some subset $E' \subseteq \calE_m$.

Recall that a \defn{sign-reversing involution} $\sigma: E' \to E'$ is a map such that for every matrix $M \in E'$,
\begin{itemize}
    \item[(i)] $\sigma(\sigma(M)) = M$ (i.e., $\sigma$ is an involution), and
    \item[(ii)] ${\tt sign}(\sigma(M)) = -{\tt sign}(M)$ (i.e., $\sigma$ is sign-reversing).
\end{itemize}
The existence of such a sign-reversing involution immediately implies
\[
\sum_{M \in E'} (-1)^{|M|} = 0,
\]
which might be helpful in calculating the alternating summation in~\eqref{eq:t-m-P-n-k-t-in-Z}, and thus in calculating $P_{n,k}$.

Sign-reversing involutions have seen many uses in algebraic combinatorics. It has applications, for example, in the theory of partitions \cite{Pak03, GL10}, in the enumerative combinatorics of trees and parking functions \cite{CP10}, and in graph theory \cite{Gri21}, to name a few. We recommend {\em Chapter~2. Counting with Signs} of the recent book of Sagan's~\cite{Sag20} for a discourse on this technique.

Let us now illustrate with an example. Consider the subset $\calE^\circ \subseteq \calE_{\binom{n}{2}}$ of all matrices $M \in \calE_{\binom{n}{2}}$ in which every row has at least $1$ one and at most $\binom{n-2}{k-2} - 1$ ones. Let $\sigma: \calE^\circ \to \calE^\circ$ be given by
\[
\sigma(M) := {\tt I}(n,k) - M,
\]
where the subtraction on the right-hand side is entrywise.

\begin{proposition}\label{prop:k-2-3-4-n-k-odd-inv}
If $k \equiv 2 \text{ or } 3 \pmod{4}$ and $\binom{n}{k}$ is odd, then the map $\sigma$ is a well-defined sign-reversing involution on $\calE^\circ$.
\end{proposition}
\begin{proof}
Let us check that the map is well-defined. Let $M' := {\tt I}(n,k) - M$. Since $M \in \calE^\circ \subseteq \calA'$, we have that $M' = {\tt I}(n,k) - M \preceq {\tt I}(n,k)$, and on every column of $M'$ the number of ones is in $\left[ 1, \binom{k}{2} - 1 \right]$. This shows $M' \in \calA'$.

Now, observe that $|{\tt Edges}(M')| = \binom{n}{2}$, and on every row of $M'$ the number of ones is in $\left[ 1, \binom{n-2}{k-2} - 1 \right]$. This shows $M' \in \calE^\circ$, and therefore $\sigma: \calE^\circ \to \calE^\circ$ is a well-defined map.

That $\sigma$ is an involution is clear from its definition:
\[
\sigma(\sigma(M)) = {\tt I}(n,k) - \left( {\tt I}(n,k) - M \right) = M.
\]

Finally, note that
\[
{\tt sign}(\sigma(M)) = (-1)^{\binom{n}{k}\binom{k}{2} - |M|} = - {\tt sign}(M),
\]
where in the last step we use that $\binom{n}{k}$ and $\binom{k}{2}$ are both odd. Hence, $\sigma$ is sign-reversing.
\end{proof}

\begin{remark}
For every positive integer $k$ (and a fortiori every $k \equiv 2 \text{ or } 3 \pmod{4}$), there are infinitely many positive integers $n \ge k$ such that $\binom{n}{k}$ is odd. More precisely, Lucas's theorem (see e.g.~\cite{Fin47, Gra97}) implies that $\binom{n}{k}$ is odd if and only if when we express $n$ and $k$ in their binary representations, every binary bit of $n$ is at least the corresponding bit of $k$.
\end{remark}

From our discussions above, we conclude that $\sum_{M \in \calE^\circ} (-1)^{|M|} = 0$, and thus the coefficient of $t^{n(n-1)/2}$ in $P_{n,k}(t)$ can be expressed as a smaller summation
\begin{equation}\label{eq:sum-M-E-circ-smaller}
\sum_{M \in \calE_{n(n-1)/2} \setminus \calE^\circ} (-1)^{|M|},
\end{equation}
whenever $n$ and $k$ satisfy the number theoretic assumptions in Proposition~\ref{prop:k-2-3-4-n-k-odd-inv}.

\subsubsection{Vanishing Coefficients at Low Degrees}
For every $\fg \in \calA$, the set of edges of $\fg$ (whose definition we give in Subsubsection~\ref{subsubsection: P-n-k-formula}) is a subset of $\binom{[n]}{2}$. From this, we deduce that if $P_{n,k}(t)$ is not the zero polynomial its degree is at most $\frac{n(n-1)}{2}$. To say this differently, for any $m > \frac{n(n-1)}{2}$, we have $[t^m] \, P_{n,k}(t) = 0$.

We have a similar result at low degrees.

\begin{proposition}
For any nonnegative integer $m$ with
\begin{equation}\label{ineq:m-less-than-n-2}
m < \frac{n^2}{2(k-1)} - \frac{n}{2},
\end{equation}
we have $[t^m] \, P_{n,k}(t) = 0$.
\end{proposition}
\begin{proof}
This is a consequence of Tur\'{a}n's theorem (see e.g. \cite[Ch.~1]{Zha23}). Take an arbitrary assignment $\fg \in \calA$. Consider an auxiliary graph $G = G(\fg)$ in which ${\bf V}(G) = [n]$ and ${\bf E}(G) = {\tt Edges}(\fg)$. Note from the definition of $\calA$ (see Subsubsection~\ref{subsubsection: P-n-k-formula}) that for any $S \in \binom{[n]}{k}$, there exists at least one edge $e \in {\tt Edges}(\fg)$ such that $e \subseteq S$. This implies that the complement graph $\ol{G}$ does not contain a clique of size $k$. Hence, by Tur\'{a}n's theorem,
\[
|{\bf E}(\ol{G})| \le \left( 1 - \frac{1}{k-1} \right) \frac{n^2}{2},
\]
and thus $|{\bf E}(G)| \ge \frac{n^2}{2(k-1)} - \frac{n}{2}$. This means that for every $\fg \in \calA$, we have
\[
|{\tt Edges}(\fg)| \ge \frac{n^2}{2(k-1)} - \frac{n}{2},
\]
which gives a lower bound on the exponent in~\eqref{eq:P-n-k-t-defn}. The desired result follows.
\end{proof}

The proposition above shows that $P_{n,k}(t)$ always has a factor $t^{\left\lceil a \right\rceil}$, where $a$ is the upper bound on the right-hand side of~\eqref{ineq:m-less-than-n-2}. In particular, $P_{n,k}(0) = 0$, for any $n \ge k \ge 2$.

\subsection{Enumerating Cliques and Anticliques}
Before we end this section, we present another use of the function $f({\bf x}|_S)$, as defined in~\eqref{eq:f-1-prod-prod}, for counting cliques and anticliques.

Throughout this subsection, we fix ${\bf x} \in \{0,1\}^{\binom{[n]}{2}}$, which corresponds to a graph $G$ on the vertex set $[n]$.

To describe the result of this subsection, we define the notion of {\em effective vertices} as follows. Suppose that $H$ is a graph on the vertex set $[n]$. We say that a vertex $V \in [n]$ if \defn{effective} if it is incident to at least one edge in ${\bf E}(H)$. In other words, $V$ is effective if its graph-theoretic degree ${\rm deg}(V)$ is at least $1$. We denote the set of effective vertices of $H$ by ${\bf V}^{\tt eff}(H)$.

For any nonnegative integers $a$ and $b$, let $\phi(a,b;G)$ denote the number of graphs $H$ on the vertex set $[n]$ such that
\begin{itemize}
    \item[(i)] ${\bf E}(H) \subseteq {\bf E}(G)$ (i.e., $H$ is a subgraph of $G$),
    \item[(ii)] $|{\bf E}(H)| = a$ (i.e., $H$ has exactly $a$ edges), and
    \item[(iii)] $|{\bf V}^{\tt eff}(H)| = b$ (i.e., $H$ has exactly $b$ effective vertices).
\end{itemize}
Let $\Phi(G)$ denote the total number of $k$-cliques and $k$-anticliques in $G$; in other words,
\[
\Phi(G) := \#\left\{ S \in \binom{[n]}{k} \, : \, G_S \text{ is a clique or an anticlique} \right\},
\]
where $G_S$ denotes the induced subgraph of $G$ on the vertex set $S$.

We have the following counting formula.

\begin{proposition}
If $k \equiv 2 \text{ or } 3 \pmod{4}$, we have
\[
\Phi(G) = \sum_{a=0}^{\binom{k}{2} - 1} \sum_{b = 0}^k (-1)^a \binom{n-b}{k-b} \phi(a,b;G).
\]
\end{proposition}
\begin{proof}
In Observation~\ref{obs:enum-form-x-0-1}, we noted that $f({\bf x}|_S) = 0$ if and only if $G_S$ is either a clique or an anticlique, and that $f({\bf x}|_S) = 1$ otherwise. Then we argued that $\prod_{S \in \binom{[n]}{k}} f({\bf x}|_S) = 1$ if and only if $G$ is a $k$-Ramsey graph, and that the product is $0$ otherwise.

Now we use the same function $f$ in a different way. Note that $1 - f({\bf x}|_S) = 1$ if and only if $G_S$ is either a clique or an anticlique. Therefore,
\[
\Phi(G) = \sum_{S \in \binom{[n]}{k}} (1 - f({\bf x}|_S)).
\]
From~\eqref{eq:f-x-S-sum-U-1-U}, we have the formula
\[
1 - f({\bf x}|_S) = \sum_{\calU \subsetneq \binom{S}{2}} (-1)^{|\calU|} x_\calU.
\]
Therefore,
\[
\Phi(G) = \sum_{S \in \binom{[n]}{k}} \sum_{\calU \subsetneq \binom{S}{2}} (-1)^{|\calU|} x_\calU.
\]
Now we note that whenever $\calU$ contains an edge $e \notin {\bf E}(G)$, the product $x_\calU$ is zero. Thus, we can add the restriction $\calU \subseteq {\bf E}(G)$ to the inner summation. Thinking of $\calU$ as the edge set of a subgraph $H$ of $G$, and switching the order of summations, we obtain
\[
\Phi(G) = \sum_H (-1)^{|{\bf E}(H)|} \sum_{\substack{S \in \binom{[n]}{k} \\ {\bf E}(H) \subsetneq \binom{S}{2}}} 1,
\]
where the outer summation is over all graphs $H$ on $[n]$ such that ${\bf E}(H) \subseteq {\bf E}(G)$. Now let us consider the inner summation, which for convenience let us denote by $\psi(H)$.

Observe that $\psi(H)$ counts the number of $S \in \binom{[n]}{k}$ such that ${\bf E}(H) \subsetneq \binom{S}{2}$. Hence, if $|{\bf E}(H)| \ge \binom{k}{2}$, then $\psi(H)$ is zero. Now, if $|{\bf E}(H)| \le \binom{k}{2} - 1$, then ${\bf E}(H) \subsetneq \binom{S}{2}$ holds if and only if $S$ contains all the effective vertices of $H$. Moreover, if $|{\bf V}^{\tt eff}(H)| = b$, then the number of such $S \in \binom{[n]}{k}$ is $\binom{n-b}{k-b}$. (Here, we use the convention that $\binom{n-b}{k-b} = 0$ if $b > k$.) From this discussion, we find
\[
\Phi(G) = \sum_{a=0}^{\binom{k}{2} - 1} \sum_{b=0}^k \sum_H (-1)^a \binom{n-b}{k-b},
\]
where the innermost summation is over all graphs $H$ on the vertex set $[n]$ such that (i)~${\bf E}(H) \subseteq {\bf E}(G)$, (ii)~$|{\bf E}(H)| = a$, and (iii)~$|{\bf V}^{\tt eff}(H)| = b$. Note that there are exactly $\phi(a,b;G)$ such graphs, by design. We have finished the proof.
\end{proof}

\bigskip

\section{The fourth formula}\label{s:proof-Q-n-k}
In this section, we give a proof of Theorem~\ref{thm:Q-n-k}.

\subsection{Proof of Theorem~\ref{thm:Q-n-k}}
As before, a graph on the vertex set $[n]$ is represented by an assignment ${\bf x} \in \{0,1\}^{\binom{[n]}{2}}$, ${\bf x}: e \mapsto x_e$. Let us perform a change of variables. We introduce ${\bf y} \in \{-1,1\}^{\binom{[n]}{2}}$, ${\bf y}: e \mapsto y_e$, given simply by $y_e := 2x_e - 1$ for every edge $e \in \binom{[n]}{2}$. Note that it is easy to go between ${\bf x}$ and ${\bf y}$, and so below we might think of ${\bf x}$ as a function of ${\bf y}$, and vice versa.

The main idea of the proof is to apply Lemma~\ref{l:sum-prod-f-A-S} with the function
\begin{equation}\label{eq:f-x-S-1-y-e}
f\!\left( {\bf x}|_S \right) = 2^{\binom{k}{2}} - \prod_{e \in \binom{S}{2}} (1+y_e) - \prod_{e \in \binom{S}{2}} (1-y_e),
\end{equation}
the ring $\calR = \mathbb{Z}$, and the subset $P = \mathbb{Z}_{\ge 1}$ (cf. our argument at the beginning of Subsection~\ref{subsection:proof-thm-P-n-k}).

In a similar fashion to Observation~\ref{obs:enum-form-x-0-1}, we obtain an enumeration formula from this function $f$ as follows. We have
\begin{equation}\label{eq:2-binom-k-2-n-k-enum}
2^{\binom{k}{2}\binom{n}{k}} \cdot \#\left\{k\text{-Ramsey graphs on }[n]\right\} = \sum_{{\bf y}} \prod_{S \in \binom{[n]}{k}} f({\bf x}|_S),
\end{equation}
where the summation is over all ${\bf y} \in \{-1,1\}^{\binom{[n]}{2}}$.

Observe that the expression in~\eqref{eq:f-x-S-1-y-e}, after we expand the products on the right-hand side, can be expressed as $2^{\binom{k}{2}} - 2 - 2 \sum y_\calU$, where the summation is over all subsets $\calU \subseteq \binom{S}{2}$ such that $\varnothing \neq \calU \subseteq \binom{S}{2}$ and $|\calU|$ is even. Here, we use the shorthand $y_\calU := \prod_{e \in \calU} y_e$.

For convenience, let us write $\tau := 1 - 2^{\binom{k}{2} - 1}$. Combining the discussion from the previous paragraph and~\eqref{eq:2-binom-k-2-n-k-enum}, we find that
\begin{equation}\label{eq:-1-binom-n-k-tau}
(-1)^{\binom{n}{k}} 2^{\left( \binom{k}{2} - 1 \right) \binom{n}{k}} \cdot \#\left\{k\text{-Ramsey graphs on }[n]\right\} = \sum_{{\bf y}} \prod_S \left( \left( \sum_\calU y_\calU \right) + \tau \right),
\end{equation}
where
\begin{itemize}
    \item the outer summation is over all ${\bf y} \in \{-1,1\}^{\binom{[n]}{2}}$,
    \item the product is over all $S \in \binom{[n]}{k}$, and
    \item the inner summation is over all nonempty subsets $\calU \subseteq \binom{S}{2}$ such that $|\calU|$ is even.
\end{itemize}

Let us now define the following combinatorial object, which is going to be helpful when we expand the product on the right-hand side of \eqref{eq:-1-binom-n-k-tau}. We define $\whcB$ to be the set of assignments $\fh$ which associate to each $S \in \binom{[n]}{k}$ a graph $\fh_S$ on the vertex set $S$ such that $|{\bf E}(\fh_S)|$ is an even nonnegative integer. Note that the set $\whcB$ is a superset of the set $\calB$, which we defined at the beginning of Subsubsection~\ref{subsubsection: Q-n-k-formula}. The notion of {\em empty locus}, which we earlier defined on $\calB$, generalizes to $\whcB$: for each $\fh \in \whcB$, we define the \defn{empty locus} of $\fh$ as (cf.~\eqref{eq:tt-Empty-fh})
\[
{\tt Empty}(\fh) := \left\{ S \in \binom{[n]}{k} \, : \, \fh_S \text{ is an empty graph} \right\}.
\]

Now we expand the product on the right-hand side of \eqref{eq:-1-binom-n-k-tau}. Using the convention that $y_{\varnothing} = 1$, we have
\[
\prod_S \left( \left( \sum_\calU y_\calU \right) + \tau \cdot y_{\varnothing} \right) = \sum_{\fh \in \whcB} \tau^{|{\tt Empty}(\fh)|} \prod_{S \in \binom{[n]}{k}} y_{{\bf E}(\fh_S)}.
\]

Therefore,
\begin{equation}\label{eq:sum-bf-y-prod-empty}
\sum_{{\bf y}} \prod_S \left( \left( \sum_\calU y_\calU \right) + \tau \cdot y_{\varnothing} \right) = \sum_{\fh \in \whcB} \tau^{|{\tt Empty}(\fh)|} \sum_{{\bf y}} \prod_{S \in \binom{[n]}{k}} y_{{\bf E}(\fh_S)}.
\end{equation}

Since the inner summation on the right-hand side of~\eqref{eq:sum-bf-y-prod-empty} is over all ${\bf y} \in \{-1, 1\}^{\binom{[n]}{2}}$, we see that
\[
\sum_{{\bf y}} \prod_{S \in \binom{[n]}{k}} y_{{\bf E}(\fh_S)} = 0
\]
whenever there exists an edge $e' \in \binom{[n]}{2}$ for which the highest power of $y_{e'}$ which divides the monomial $\prod_{S \in \binom{[n]}{k}} y_{{\bf E}(\fh_S)}$ has an odd exponent; in other words, the summation is zero if there exists $e'$ such that the integer $\varpi_{e'}$ is odd in the expansion
\[
\prod_{S \in \binom{[n]}{k}} y_{{\bf E}(\fh_S)} = \prod_{e \in \binom{[n]}{2}} y_e^{\varpi_e}.
\]
Here, for each $e \in \binom{[n]}{2}$, the integer $\varpi_e$ has the formula
\[
\varpi_e = \#\left\{ S \in \binom{[n]}{k} \, : \, e \in {\bf E}(\fh_S) \right\}.
\]

This means that the contribution to the summation on the right-hand side of~\eqref{eq:sum-bf-y-prod-empty} from the term with $\fh \in \whcB$ is zero whenever $\fh \in \whcB - \calB$. Therefore,
\begin{equation}\label{eq:RHS-B-hat-B}
{\rm RHS}_{\eqref{eq:sum-bf-y-prod-empty}} = \sum_{\fh \in \calB} \tau^{|{\tt Empty}(\fh)|} \sum_{{\bf y}} \prod_{S \in \binom{[n]}{k}} y_{{\bf E}(\fh_S)}.
\end{equation}
On the other hand, for any $\fh \in \calB$ and for any ${\bf y} \in \{-1, 1\}^{\binom{[n]}{2}}$, we have
\[
\prod_{S \in \binom{[n]}{k}} y_{{\bf E}(\fh_S)} = 1.
\]

Now we combine the last observation with~\eqref{eq:-1-binom-n-k-tau}--\eqref{eq:RHS-B-hat-B} and the definition of $Q_{n,k}$ from~\eqref{eq:Q-n-k-t-defn} to obtain
\[
(-1)^{\binom{n}{k}} 2^{\left( \binom{k}{2} - 1 \right) \binom{n}{k}} \cdot \#\left\{k\text{-Ramsey graphs on }[n]\right\} = 2^{\binom{n}{2}} \cdot Q_{n,k}(\tau),
\]
which implies Theorem~\ref{thm:Q-n-k}, from which Corollary~\ref{cor:Q-n-k} is immediate.

\subsection{Matrix Formulation}
Just as in Subsubsection~\ref{subsubsection: P-n-k-matrix}, we can give a matrix formulation for the polynomials $Q_{n,k}$. We start with the same incidence matrix ${\tt I}(n,k)$ as in~\eqref{in:I-n-k}, and let $\preceq$ denote the same partial order on $\{0,1\}^{\binom{[n]}{2} \times \binom{[n]}{k}}$ as before (see Subsubsection~\ref{subsubsection: P-n-k-matrix}).

For each matrix $M \in \{0,1\}^{\binom{[n]}{2} \times \binom{[n]}{k}}$, for each $e \in \binom{[n]}{2}$, and for each $S \in \binom{[n]}{k}$, let
\[
|M_{e,-}| := \sum_{S \in \binom{[n]}{k}} M_{e,S},
\]
and
\[
|M_{-,S}| := \sum_{e \in \binom{[n]}{2}} M_{e,S}.
\]

Let $\calB'$ denote the set of all matrices $M \preceq {\tt I}(n,k)$ such that for any $e \in \binom{[n]}{2}$ and for any $S \in \binom{[n]}{k}$, the integers $|M_{e,-}|$ and $|M_{-,S}|$ are even. For each $M \in \calB'$, define
\[
{\tt Empty}(M) := \#\left\{ S \in \binom{[n]}{k} \, : \, |M_{-,S}| = 0 \right\}.
\]
Then the formula
\begin{equation}\label{eq:Q-n-k-t-matrix-M}
Q_{n,k}(t) = \sum_{M \in \calB'} t^{|{\tt Empty}(M)|}
\end{equation}
follows by design, since we have the simple one-to-one correspondence which maps each $\fh \in \calB$ to $M \in \calB'$ according to the following rule: $M_{e,S} = 1$ if and only if $e \in {\bf E}(\fh_S)$ (see Subsubsection~\ref{subsubsection: Q-n-k-formula}).

Note that every coefficient of $Q_{n,k}(t)$ is a non-negative integer. The polynomial $Q_{n,k}(t)$ has degree exactly $\binom{n}{k}$ with leading coefficient $1$, since there exists a unique matrix $M \in \calB'$ with ${\tt Empty}(M)$ of {\em maximum} size $\binom{n}{k}$: the zero matrix.

\bigskip

\section{Ideas for Further Investigation}\label{sec:ideas-for-further-investigation}
In this section, we describe some ideas for further investigation.

\subsection{Properties of \texorpdfstring{$P_{n,k}$}{Pnk} and \texorpdfstring{$Q_{n,k}$}{Qnk}}
In this work, we have defined the univariate polynomials $P_{n,k}(t)$ and $Q_{n,k}(t)$, and shown that they are related to diagonal Ramsey numbers. These two families of polynomials might be interesting on their own, even outside the context of Ramsey theory. For example, what can we say further about their algebraic properties?

\subsection{Sign-Reversing Involutions}
Proposition~\ref{prop:k-2-3-4-n-k-odd-inv} gives one example of a sign-reversing involution from which we deduce the formula in~\eqref{eq:sum-M-E-circ-smaller} for computing a coefficient in $P_{n,k}(t)$ under some number theoretic assumptions. As we mentioned in Subsubsection~\ref{subsubsection: sign-reversing}, there seem to be many interesting sign-reversing involutions which might help simplify computing $P_{n,k}(t)$. Similar techniques might also help with computing the summations in our first, second, and fourth formulas as well.

\subsection{Erd\H{o}s--R\'{e}nyi Random Graphs}
Theorems~\ref{thm:P-n-k} and~\ref{thm:Q-n-k} give formulas for the probability that a uniformly random graph $G \sim \bbG(n,1/2)$ is $k$-Ramsey in terms of values of $P_{n,k}(t)$ and $Q_{n,k}(t)$. It might be interesting to find formulas with similar algebraic flavors for the analogous probability when $G$ follows $\bbG(n,p)$ with a general probability parameter $p \in [0,1]$.

\subsection{Off-diagonal Ramsey numbers and Other Variants}
We have focused on the diagonal Ramsey numbers $R(k)$. It might be interesting to find similar formulas for off-diagonal Ramsey numbers $R(k,\ell)$ as well. For a survey on $R(k,\ell)$, see~\cite{Rad21}. There are also many variants of the Ramsey problem. For instance, as many as the number of times the word {\em hypergraph} appears in our current work, we did not focus on the {\em hypergraph Ramsey numbers} (see e.g.~\cite{CFS15}). Would there be some elegant algebraic formulas one could find for other variants of the Ramsey problem?

\subsection{Using Lemma~\ref{l:sum-prod-f-A-S} with Other Commutative Rings}
Lemma~\ref{l:sum-prod-f-A-S} in Section~\ref{sec:general-algebraic} provides a rather general framework to obtain formulas for diagonal Ramsey numbers. Perhaps one could apply the lemma to different commutative rings $\calR$ to obtain further connections between Ramsey numbers and other algebraic objects.

\subsection{A Bivariate Polynomial}
Consider the following bivariate polynomial
\[
B_{n,k}(z,w) := \sum_{G,H} z^{|{\bf E}(H)|} w^{{\tt i}(G,H)} \in \mathbb{Z}[z,w],
\]
where the summation is over all graphs $G$ and all $k$-uniform hypergraphs $H$ on the vertex set $[n]$. (Recall that the definitions of ${\bf E}$ and ${\tt i}$ were given in Section~\ref{sec:defns-main-results}.) In light of Theorem~\ref{thm:cos-incidence}, we know that $B_{n,k}(z,w)$ is related to diagonal Ramsey numbers. On the other hand, this polynomial appears interesting on its own, as an object in algebraic combinatorics. For example, it is easy to see that the total degree of $B_{n,k}(z,w)$ is $\binom{n}{k} + \binom{n}{k}\binom{k}{2}$ (when $k = 2$, this is \href{https://oeis.org/A002378}{A002378} on~\cite{OEIS}; when $k = 3$, \href{https://oeis.org/A210440}{A210440}), that $B_{n,k}(1,1) = 2^{\binom{n}{2} + \binom{n}{k}}$ (when $k = 2$, \href{https://oeis.org/A053763}{A053763}; when $k = 3$, \href{https://oeis.org/A125791}{A125791}), and that $B_{n,k}(0,w) = 2^{\binom{n}{2}}$ (\href{https://oeis.org/A006125}{A006125}). What are some other interesting things we can say about this polynomial?

We have seen in Theorem~\ref{thm:cos-incidence} that the zero locus of $B_{n,k}(z,w)$ is related to Ramsey numbers. Let us ask: what are some interesting zero-free regions of $B_{n,k}$ (when considered as a holomorphic function $B_{n,k}: \mathbb{C}^2 \to \mathbb{C}$ in two variables)?

\bigskip

\section*{Acknowledgments}
P.J. was supported by Elchanan Mossel's Vannevar Bush Faculty Fellowship ONR-N00014-20-1-2826 and by Elchanan Mossel's Simons Investigator award (622132). I thank Sammy Luo and Elchanan Mossel for insightful discussions. I thank Sorawee Porncharoenwase for useful programming help. I thank Wijit Yangjit for his words of encouragement. I thank Julian Sahasrabudhe for giving an inspiring colloquium talk at MIT on the leap day of 2024.

\bibliographystyle{alpha}
\bibliography{ref}
\end{document}